\documentclass[11pt]{article}
\usepackage{max}

\newcommand{\sat}{\text{sat}}
\newcommand{\ex}{\text{ex}}
\newcommand{\F}{\mathcal{F}}
\newcommand{\M}{\mathcal{M}}
\renewcommand{\S}{\mathcal{S}}

\title{Saturation Numbers for Minors}
\author{Max Aires \\ 
Dept.\ Math.\ Sciences \\ Carnegie Mellon University \\ Pittsburgh, PA\\ 
maires@andrew.cmu.edu}

\begin{document}

\maketitle

\begin{abstract}
\noindent
The saturation number $\sat(n,\F)$ is the minimum number of edges in any graph which does not contain a member of $\F$ as a subgraph, but will if any edge is added. We give a few upper and lower bounds for saturation numbers for minors. In particular, we shall show that certain Generalized Petersen Graphs are $K^r$-minor saturated for $6\le r\le 8$.
\end{abstract}

\section{Introduction}

Let $\F$ be a family of graphs. We say a graph $G$ is $\F$-saturated if it does not contain a member of $\F$ as a subgraph, but $G+e$ does for any edge $e\notin G$. The saturation number $\sat(n, \F)$ represent the minimum number of edges in an n-vertex $\F$-saturated graph. This can be seen as a corresponding lower bound to the classical extremal number $\ex(n,\F)$. Let $\M(H)$ be the class of graphs containing $H$ as a minor. The main purpose of this paper is to derive some basic results about about $\sat(n,\M(H))$. While there has been a large amount of work investigating extremal numbers for $\M(H)$, there has been almost none on the corresponding saturation number.

Since the saturation number is upper bounded by the extremal number, a major result from Thomasson tells us that $\sat(n,\M(H))\le\ex(n,\M(H))\le (\al |V(H)|\s{\log(|V(H)|})n + O(1)$ where $\al$ is a fixed constant \cite{Th}. On the other hand, if $H$ is connected and contains at least $3$ vertices, then it follows that $\sat(n,\M(H))\ge \f{1}{2}(n-1)$, as no $\M(H)$-saturated graph can contain more than one isolated vertex (or else the edge between two isolated vertices could be added without creating a minor of $H$). Hence the main question about the saturation numbers for minors lies in determining the value of the linear constant (if such a limit even exists). One difficulty however is that for a subgraph $H'\subseteq H$, it is not necessarily true that $\sat(n,\M(H'))\le\sat(n,\M(H))$; thus it is not sufficient to study complete graphs to upper bound saturation numbers. 

While the saturation number for minors remains almost unexplored, we can nonetheless deduce some easy resulting from existing work on saturation number for other classes. Since a graph contains the path $P^r$ as a minor if and only if it contains it as a subgraph, we have $\sat(n,\M(P^r))=\sat(n,P^r)\le n+O(1)$~\cite{KT}. Now let $\S(H)$ be the family of subdivisions of $H$; it follows that $\sat(n,\M(C^r))=\sat(n,\S(C^r))$. It is known that there exists a constant $c$ so that $\f{5}{4}n\le\sat(n,\S(C^r))\le (\f{5}{4}+\f{c}{r^2})n+O(1)$~\cite{FJMTW}. Finally, $\sat(n,\M(K^3))=n-1$, $\sat(n,\M(K^4))=2n-3$, and $\sat(n,\M(K^5))=\f{11}{6}n+O(1)$; these values are easily obtained via well known characterizations of graphs which don't contain $K^3$, $K^4$, or $K^5$ as minors.

In this paper, we shall show some general upper and lower bounds for the saturation number for minors of minimum degree at least $3$. We shall also give examples of saturated graphs which give good upper bounds for $\sat(n,\M(K_{1,r}))$ and for $\sat(n,\M(K^r))$ for $6\le r\le 8$.
These examples constructions are particularly interesting as they provide a direct contrast with the saturation numbers for regular subgraphs and subdivisions. We shall also compute $\sat(n,\M(K_{3,3}))$.

\section{Main Results}

We start out with a general lower bound on $\sat(n,\M(H))$.

\begin{thm}[Lower Bounds] If $H$ has $\de(H)\ge 3$ then $\sat(n, \M(H))\ge \f{3}{2}n$ for all $n\ge 4$. Furthermore, if $H$ has $\de(H)\ge 4$ and is also triangle-free, then we can improve this to $\sat(n,\M(H))\ge 2n$ for $n\ge 5$.
\end{thm}

\begin{proof}

We shall first show that $\de(H)\ge 3$ implies $\sat(n,\M(H))\ge \f{3}{2}n$. Suppose for the sake of contradiction the claim were false, and let $G$ be a minor-minimal counterexample on $n\ge 4$ vertices and $m<\f{3}{2}n$ edges. Then the average degree of $G$ is less than $3$, so there exists a vertex $v\in G$ with $\deg(v)\le 2$. If $\deg(v)=1$, and its neighbor is $w$, then $w$ must have another neighbor, and we can connect $v$ to that neighbor and still have a $H$-minor-free graph. Now suppose $\deg(v)=2$, and that its neighbors are $x$ and $y$. If $xy$ are not adjacent, then $G+xy$ must contain an $H$ minor, so we can partition the vertices of $G$ into a connected group for each vertex in $H$ with an edge between components in $G$ for each edge in $H$. Since $\de(H)\ge 3$, $v$ can not be in its own group, so without loss of generality we can assume $x$ is in the same group as $v$. Then the edge $xy$ is redundant with the edge $vy$, so the minor also exists in $G$. Hence $x$ and $y$ are adjacent. Then $G-v$ is also $\M(H)$-saturated, as if $(G-v)+e$ does not contain $H$ as a minor, $G+e$ will not either, as no minor can have $v$ in its group. So $G-v$ has $n-1$ vertices and $m-2<\f{3}{2}(n-1)$ edges, so $G-v$ is a smaller counterexample. Hence we have show the first claim.

We now prove that $\sat(n,\M(H))$ if $\de(H)\ge 4$ and $H$ is triangle-free. Again let $G$ be a minor-minimal counterexample with $n\ge 5$ vertices and $m<2n$ edges; then there exists $v\in G$ so that $\deg(v)\le 3$. The cases of $\deg(v)=1$ and $2$ are as before, so suppose $\deg(v)=3$. Suppose that $v$ has neighbors $x,y,z$ and that $x$ and $y$ are not adjacent. Then $G+xy$ must contain a minor of $H$. Consider a partition of the vertices of $G$ which shows this, and let $v$ belong to $U$; since $\deg(H)\ge 4$, one of the neighbors of $U$ must as well. If $x$ or $y$ is part of $U$, the the minor also exists in $G$ (since the edge $xy$ will be redundant with $vx$ and $vy$). So $z$ must be part of $U$. If $v$ and $x$ are both part of group $V$, then moving $v$ to group $V$ and removing $xy$ shows the minor also appears in $G$. Otherwise, $x$ appears in $V_1$ and $y$ appears in $V_2$. Then the edge $V_1V_2$ must appear in $H$, so one of $V_1U$ or $V_2U$ does not (since the graph is triangle-free). If $V_1U$ does not appear then moving $v$ to $V_2$ again creates the minor in $G$.
\end{proof}

The example of $\sat(n,\M(K^5))=\f{11}{6}n+O(1)$ shows the the triangle-free requirement is necessary in the case of $\de(H)=4$; our upper bound for $\sat(n,\M(K^r))$ for $6\le r\le 8$ will show it is also necessary for $5\le \de(H)\le 7$.

Next we show a general upper bound.

\begin{thm}[Upper Bound]
Let $H$ be connected with $d:=\de(H)\ge 3$ and $s:=|V(H)|$. Then
$$\sat(n, \M(H))\le \ll((d-1)+\f{\bn{s-1}{2}+\bn{d}{2}-(d-1)(s-1)}{\bn{s-1}{d-1}+(s-1)-d}\rr)n+O(1)$$
\end{thm}

\begin{proof} 
Create a graph $G$ by starting with $K^{s-1}$ and, for each subset of size $d-1$ except one, connecting a new vertex to every element in the set. Observe that $|V(G)|=(s-1)+\binom{s-1}{d-1}-1$ and $|E(G)|=\bn{s-1}{2}+(d-1)\bn{s-1}{d-1}-(d-1)$. $G$ does not contain $H$ as a minor, as only $s-1$ vertices of $G$ have degree at degree at least $d$, and contracting any of the edges edges not in the original $K^{s-1}$ cannot create an additional vertex of degree $d$. Now observe that $G$ is $\M(H)$-saturated: adding any edge will allow us to contract to a graph of $K^{s-1}$ with an additional vertex connected to $d$ elements; this clearly has $H$ as a subgraph.

We now create larger $\M(H)$-saturated graphs as follows. If $\ka(H)=1$, take disjoint unions of $G$. If $2\le \ka(H)\le d-1$, union multiple copies of $G$ along any common $K^{\ka(H)-1}$ within the $K^{s-1}$. If $\ka(H)=d$, then union them along the $d-1$ vertices which were not connected to an added vertex. Since we are gluing along common complete graphs of size at most $\ka(H)-1$, any $H$ minor must lie entirely within one copy of $G$, which is impossible. So this combined graph does not have any copies of $G$ as a minor. If we add an edge within one of the copies of $G$, we will create a minor of $H$ (since $G$ is $\M(H)$-saturated. If we add an edge between two different copies of $G$ that doesn't include a vertex in the common set, then we can contract this down to adding an additional edge to one of the copies of $G$. So the combined graph is also $\M(H)$-saturated.

So
\begin{align*}\sat(n, \M(H))&\le \f{\bn{s-1}{2}+(d-1)\bn{s-1}{d-1}-(d-1)-\bn{d-1}{2}}{(s-1)+\binom{s-1}{d-1}-1-(d-1)}n+O(1)\\
&= \ll((d-1)+\f{\bn{s-1}{2}+\bn{d}{2}-(d-1)(s-1)}{\binom{s-1}{d-1}+(s-1)-d}\rr)n+O(1)
\end{align*}
\end{proof}

In particular, the above bound shows that $\sat(n,\M(H))\le \de(H) n + O(1)$ when $\de(H)\ge 3$.

As an example of this upper bound, consider the graph $K_{3,3}$. By a theorem of Wagner, every edge-maximal graph with no $K_{3,3}$ minor can be constructed by recursively combining $K^5$ and maximal planar graphs along a common edge \cite{Wa}. This gives us that $\sat(n,\M(K_{3,3,}))= \f{9}{4}n+O(1)$. This matches the upper bound from Theorem 2.2.

We shall now consider the saturation number for $\M(K_{1,r})$.

\begin{prop}
$\sat(n,\M(K_{1,r}))\le n+O_r(1)$
\end{prop}

\begin{proof}
Take a $K^r$ with vertices $u_1,\dots,u_r$ and replace the edge $u_1u_2$ with an arbitrary long path. Suppose there is a $K^r$ minor in this graph; then there exist connected sets of vertices $U$ and $V_1,\dots,V_r$ so that there is an edge between each pair ($U$, $V_i$). At most two of the $V_i$ can come from along the path, so the other $r-2$ must be within the original $u_1,\dots,u_r$, meaning $V$ can only contain two elements from among these. If $U$ consists of a single element from the $u_i$, it can only be adjacent to at most $r-1$ other sets. If $U$  contains two elements, it can not contain both $u_1$ and $u_2$, so only one $V_i$ can appear along the path, so again $U$ can only be adjacent to $r-1$ other sets. Hence this graph does not contain $K_{1,r}$ as a minor.

Adding an edge that is incident to any of the $u_i$ means that $u_i$ will now have degree $r$. Otherwise, adding an edge between two vertices along the path will allow the graph to be contracted down to $K^r$ with an extra path of between $u_1,u_2$, so $u_1$ will have degree $r$. Hence this graph is $\M(K_{1,r})$-saturated.
\end{proof}

Finally, we consider complete graphs. Note that by Theorem 2.1, $\sat(n,MK^r)\ge \f{3}{2}n$ for $r\ge 4$. It turns out that, at least for $6\le r\le 8$, this lower bound is fairly close.

\begin{thm}
The following bounds hold:
$$\sat(n,\M(K^6))\le \f{23}{14}n+O(1)$$
$$\sat(n,\M(K^7))\le \f{19}{12}n+O(1)$$
$$\sat(n,\M(K^8))\le \f{14}{9}n+O(1)$$
\end{thm}

\begin{proof}
For all three upper bounds, we shall repeatedly union together multiple copies of a triangle-free $\M(K^r)$-saturated graph along a common edge. Note that such a graph cannot contain a $K^r$ minor, and must be $\M(K^r)$-saturated since any added edge can be contracted down to an added edge in one of the copies of the original graph. 

Let $G(n,k)$ be the graph with vertices $x_0,\dots, x_{n-1}$ and $y_0,\dots,y_{n-1}$ so that $x_i\sim x_{i+1}$, $x_i\sim y_i$, and $y_i\sim y_{i+k}$, where indices are taken modulo $k$. These graphs are known as Generalized Petersen Graphs. We claim that $G(8,3)$, $G(13,5)$, and $G(19,7)$ are $\M(K^r)$-saturated for $r=6,7$, and $8$ respectively. Note that unioning these along a common edge will give the desired upper bounds.

Firstly, suppose that $G(8,3)$ contained $K^6$ as a minor. Then we could divide the $16$ vertices into $6$ connected sets with an edge between each two pairs of sets. There would need to be $16-6=10$ edges within the connected components, and there would need to be $\binom{6}{2}=15$ edges between components; since there are only $24$ total edges, this is impossible. The same argument works for the $r=7$ and $8$ cases. 

It remains to show that adding any edge to these graphs will create the desired minor. Note first that due to the symmetry of Generalized Peterson Graphs, $G+x_0x_j$ and $G+x_ix_{i\pm j}$ are isomorphic for all $i,j$, and similarly for $G+x_0y_j$ and $G+y_0y_j$; this observation considerably cuts down the number of edges we must consider. We shall now show some divisions of the vertices into connected groups so that every pair of groups contain a pair of adjacent vertices but one. Consider the groupings below:

\begin{center}

\begin{tabular}{ccc}
\begin{tabular}[t]{|c|}\hline
$G(8,3)$\\\hline
$A_1: x_3, x_4, x_5$\\\hline
$A_2: y_1, y_3, y_6$\\\hline
$A_3: y_2, y_4, y_7$\\\hline
$A_4: x_6, x_7, x_0$\\\hline
$A_5: x_1, x_2$\\\hline
$A_6: y_0, y_5$\\\hline
\end{tabular}
&
\begin{tabular}[t]{|c|}\hline
$G(13,5)$\\\hline
$A_1: x_8,y_0,y_3,y_8$\\\hline
$A_2: x_9,y_9,y_4,y_1$\\\hline 
$A_3: x_{10},y_2,y_5,y_{10}$\\\hline
$A_4: x_0,x_1,x_2,x_{12}$\\\hline
$A_5: x_6,x_{11},y_6,y_{11}$\\\hline
$A_6: x_3,x_4,x_5$\\\hline
$A_7: x_7,y_7,y_{12}$\\\hline
\end{tabular}
&
\begin{tabular}[t]{|c|}\hline
$G(19,7)$\\\hline
$A_1: x_0, x_1, x_{18}, y_1,y_8$\\\hline
$A_2: x_2, x_3, x_4, y_2, y_3$\\\hline
$A_3: x_5, x_6, y_{6}, y_{13}, x_{13}$\\\hline
$A_4: x_{10}, y_{10}, y_{17}, x_{17},y_{5}$\\\hline
$A_5: x_{12}, y_0,y_7,y_{12},y_{14}$\\\hline
$A_6: x_{14}, x_{15}, x_{16}, y_{15},y_{16}$\\\hline
$A_7: x_7, x_8,x_9, y_{9}$\\\hline
$A_8: x_{11},y_{4}, y_{11}, y_{18}$\\\hline
\end{tabular}
\end{tabular}
\end{center}

First, consider $G(8,3)$. Every pair of groups contain adjacent vertices except for $A_5=\{x_1,x_2\}$ and $A_6=\{y_0,y_5\}$, which means there will be a $K^6$ minor if we add the edges $(x_1y_0\cong)x_0y_1$, $(x_2y_0\cong)x_0y_2$, $(x_2y_5\cong)x_0y_3$, or $(x_1y_5\cong)x_0y_4$. Furthermore, since $x_0$ and $y_2$ are adjacent to both $A_5$ and $A_6$, removing either from its respective groups and adding them to either $A_5$ or $A_6$ will connect the two and create a new pair of groups which is disconnected, so it will create different groupings with all but one pair of groups connected. By moving $x_0$ from $A_4$ to $A_6$ we will have a partition where every pair of groups is connected except for $A_4\bs\{x_0\}=\{x_6,x_7\}$ and $A_5=\{x_1,x_2\}$. So we cannot add $(x_7x_1\cong)x_0x_2$, $(x_6x_1\cong)x_0x_3$, or $(x_6x_2\cong)x_0x_4$. Finally, if you instead move $y_2$ from $A_3$ to $A_5$, then every pair is connected except $A_3\bs \{y_2\}=\{y_4,y_7\}$ and $A_6=\{y_0, y_5\}$, which eliminates the edges $(y_4y_5\cong)y_0y_1$, $(y_5y_7\cong)y_0y_2$, and $y_0y_4$. Hence we have exhausted every possible edge, so $G(8,3)$ is in fact $\M(K^6)$-saturated.

Now consider $G(13,5)$. As before, every group is connected and every pair of groups is contains an edge between them except for $A_6$ and $A_7$, and both pairs are adjacent to $y_4$ and $x_6$. From $A_6=\{x_3,x_4,x_5\}$ and $A_7=\{x_7,y_7,y_{12}\}$, we can't add $(x_5x_7\cong)x_0x_2$, $(x_4x_7\cong)x_0x_3$, $(x_3x_7\cong)x_0x_4$, $(x_5y_7\cong)x_0y_2$, $(x_4y_7\cong)x_0y_3$, $(x_3y_7\cong)x_0y_4$, $(x_4y_{12}\cong)x_0y_5$, $(x_5y_{12}\cong)x_0y_6$. By moving $x_6$ to $A_6$, every pair is connected but $A_5\bs\{x_6\}=\{x_{11},y_6,y_{11}\}$ and $A_7=\{x_7,y_7,y_{12}\}$, so we also can't add $(y_6x_7\cong)x_0y_1$, $(y_6y_7\cong)y_0y_1$, $(y_{11}y_7\cong)y_0y_4$, or $(y_6y_{12}\cong)y_0y_6$. By moving $x_6$ to $A_7$ instead, we can't connect $A_5\bs \{x_6\}=\{x_{11},y_6,y_{11}\}$ and $A_6=\{x_3,x_4,x_5\}$, which bans $(x_{11}x_3\cong)x_0x_5$ and $(x_{11}x_4\cong)x_0x_6$. Finally, by instead moving $y_4$ to $A_7$, every pair is connected except $A_2\bs \{y_4\}=\{x_9,y_9, y_1\}$ and $A_7=\{x_7,y_7,y_{12}\}$, we also can't add $(y_9y_7\cong)y_0y_2$ or $(y_9y_{12}\cong)y_0y_{3}$. Again, we have exhausted every possible edge, so $G(13,5)$ is $\M(K^7)$-saturated.

Finally, consider $G(19,7)$. Once again, every group is connected and every pair contains an edge between them except for $A_7$ and $A_8$, which are also both adjacent to $x_{10}$ and $y_{16}$. From $A_7=\{x_7,x_8,x_9,y_9\}$ and $A_8=\{x_{11},y_{4},y_{11},y_{18}\}$, we can not add $(x_9x_{11}\cong)x_0x_2$, $(x_8x_{11}\cong)x_0x_3$, $(x_7x_{11}\cong)x_0x_4$, $(x_9y_{11}\cong)x_0y_2$, $(x_8y_{11}\cong)x_0y_3$, $(x_7y_{11}\cong)x_0y_4$, $(x_9y_4\cong)x_0y_5$, $(x_7y_{18}\cong)x_0y_8$, $(x_8y_{18}\cong)x_0y_9$, $(y_9y_{11}\cong)y_0y_2$, $(y_9y_4\cong)y_0y_5$, $(y_9y_{18}\cong)y_0y_9$. By moving $y_{16}$ from $A_6$ to $A_7$, every pair of sets is connected except $A_6\bs\{y_{16}\}=\{x_{14},x_{15},x_{16},y_{15}\}$ and $A_8=\{x_{11},y_{4},y_{11},y_{18}\}$, which eliminates the edges $(x_{16}x_{11}\cong)x_0x_5$, $(x_{16}y_4\cong)x_0y_7$, $(y_{15}y_{18}\cong)y_0y_3$, $(y_{15}y_{11}\cong)y_0y_4$, $(y_{15}y_4\cong)y_0y_8$. By moving $y_{16}$ from $A_6$ to $A_8$ instead we have the sets $A_6\bs \{y_{16}\}=\{x_{14},x_{15},x_{16},y_{15}\}$ and $A_7=\{x_7,x_8,x_9,y_9\}$. This bans $(x_{15}x_9\cong)x_0x_6$, $(x_{15}x_8\cong)x_0x_7$, $(x_{15}x_7\cong)x_0x_8$, $(x_{16}x_7\cong)x_0x_9$, $(x_{15}y_{9}\cong)x_0y_6$, and $(y_{15}y_9\cong)y_0y_6$.
Finally, by moving $x_{10}$ from $A_4$ to $A_7$, we have $A_4\bs\{x_{10}\}=\{x_{17},y_{10},y_{17},y_5\}$ and $A_8=\{x_{11},y_{4},y_{11},y_{18}\}$; this bans $(x_{17}y_{18}\cong)x_0y_1$ and $(y_{17}y_{18}\cong)y_0y_1$.
\end{proof}

Note that our construction actually shows $\sat(16,\M(K^6))=24$, and similar values for $K^7$ and $K^8$.
Our construction here is rather similar to the case of $K^5$, where $\M(K^5)$-saturated graphs are constructed by gluing copies of the Wagner graph along a common edge. However, it seems difficult to raise the corresponding lower bound, as the class of graphs with no $K^6$ minor has much greater variety than that of $K^5$ and is still not very well understood.

\section{Conclusion}

Our upper bounds for $\sat(n,\M(K_{1,r}))$ and $\sat(n,\M(K_r))$ provide some interested comparisons with traditional saturation numbers. While $\sat(n, \M(K_{1,r}))$ is at most $n+O(1)$, the traditional $\sat(n,K_{1,r})=\f{r-1}{2}n+O(1)$, as is $\sat(n,\S(K_{1,r}))$. Similarly, while $\sat(n,K^r)=(r-2)n+O(1)$, the values for the corresponding minors are less than $2$ and decreasing for $5\le r\le 8$. It does not seem to be known even whether for all $C$, there exist graphs with $\sat(n,\M(H))\ge Cn+O(1)$, or whether some linear constant $C$ holds as an upper bound for all minors. The author feels this question merits further investigation.


\begin{thebibliography}{1}

\bibitem{FJMTW}
M.~Ferrara, M.~Jacobson, K.~G. Milan, C.~Tennenhouse, and P.~S. Wenger.
\newblock Saturation numbers for families of graph subdivisions.
\newblock {\em Journal of Graph Theory}, 71(4), 2012.

\bibitem{KT}
L.~Kászonyi and T.~Tuza.
\newblock Saturated graphs with minimal number of edges.
\newblock {\em Journal of Graph Theory}, 10(2), 1986.

\bibitem{Th}
A.~Thomason.
\newblock Extremal numbers for complete minors.
\newblock {\em Journal of Combinatorial Theory, Series B}, 81, 2001.

\bibitem{Wa}
K.~Wagner.
\newblock Uber eine {E}rweiterung des {S}atzes von {K}uratowski.
\newblock {\em Deutsche Mathematik}, 2, 1937.

\end{thebibliography}

\end{document}